\documentclass[11pt]{article}

\usepackage{amsthm,amsmath}
\usepackage{fancyhdr}
\usepackage[T1]{fontenc}
\usepackage[theoremfont]{newtxtext}
\usepackage{newtxmath}
\usepackage{graphicx}
\usepackage{color}
\usepackage{mathtools}
\usepackage{geometry}
\usepackage{enumerate}
\usepackage{titlesec}
\titleformat*{\section}{\normalfont\large\bfseries}
\titleformat*{\subsection}{\normalfont\normalsize\bfseries}
\usepackage{caption}
\usepackage[%
pdftitle={A Locally Mass Conserving Quadratic Velocity, Linear Pressure Element},
pdfauthor={R.W. Thatcher and D.J. Silvester},
colorlinks=true, allcolors=blue, breaklinks]{hyperref}

\graphicspath{{.}}

\theoremstyle{plain}
\newtheorem{thm}{Theorem}[section]
\newtheorem*{cor}{Corollary}
\newcommand{\bb}[2]{b^{(#1)}_{#2}}
\newcommand{\beps}{\boldsymbol\varepsilon}
\renewcommand{\div}{\operatorname{div}}
\newcommand{\esM}{\beps^\sigma_M}

\newcommand{\Lto}{L^2_0}
\newcommand{\Ne}{\textup{N}_\textup{e}}
\newcommand{\Nn}{\textup{N}_\textup{n}}
\newcommand{\PhM}{P^h_M}
\newcommand{\PiM}{\Pi^{(1)}_M}
\newcommand{\pr}{\partial}
\newcommand{\Rey}{\mathcal{R}}
\newcommand{\RhM}{R^h_M}
\newcommand{\VhM}{V^h_M}
\newcommand{\tri}[1]{\triangle^{(#1)}}
\newcommand{\bP}{\boldmath{P}}

\definecolor{otherblue}{rgb}{0,0.3,0.6}
\newcommand\rblf[1]{\textcolor{otherblue}{#1}}

\begin{document}

\pagenumbering{gobble}
\vspace*{-45pt}
\begin{center}
\includegraphics[width=.8\textwidth]{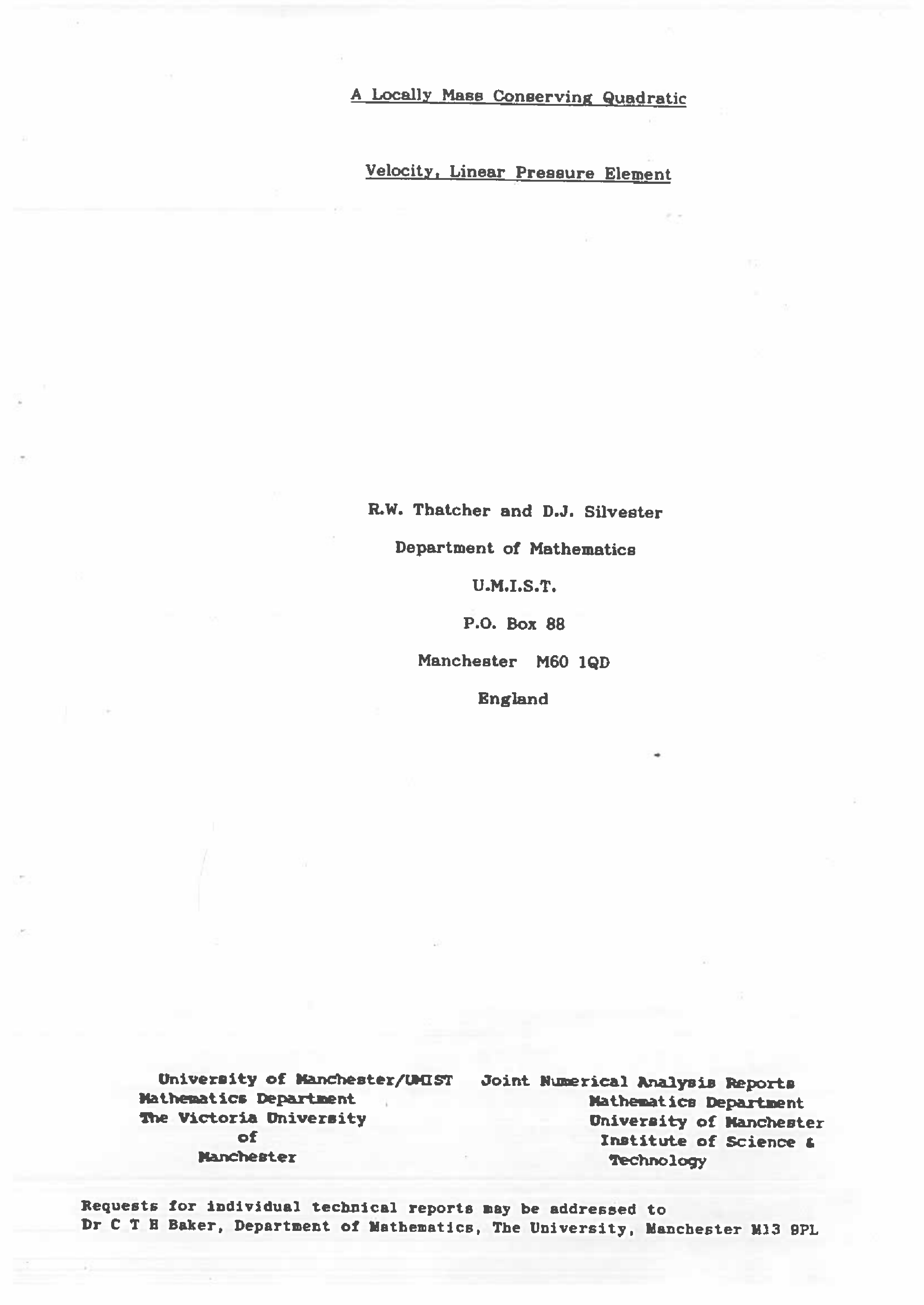}
\end{center}
\vspace*{-35pt}
\begin{abstract}
By supplementing the pressure space for the Taylor--Hood element a triangular
element that satisfies continuity over each element is produced. Making a novel
extension of the patch argument to prove stability, this element is shown to be
globally stable and give optimal rates of convergence on a wide range of
triangular grids. This theoretical result is extended in the discussion given
in the appendix, showing how optimal convergence rates can be obtained on all
grids.
  Two examples are presented, one illustrating the convergence rates and the
other illustrating difficulties with the Taylor--Hood element which are overcome
by the element presented here.
\end{abstract}

\newpage
\pagenumbering{arabic}\setcounter{page}{2}

\section{Introduction}\label{sec:1}

A popular triangular element for solving two-dimensional flows was introduced by
Hood \& Taylor~\cite{HT74}. It has the serious physical drawback that continuity is only
satisfied over the whole domain of the problem and not over each element. A consequence
of this phenomenon and the poor approximation that can result is given by Tidd, Thatcher
\& Kaye~\cite{TTK86} and a further example of poor results is illustrated in section~\ref{sec:6}.
Tidd et al. show that by supplementing the continuous linear pressure by functions
constant in each element, not only is continuity satisfied locally but also the quality
of the solution is greatly improved, at least for the particular problem that they were
considering. This idea of supplementing the pressure space had previously been suggested
by Gresho et al.~\cite{GLCL80} and is also discussed by Griffiths~\cite{Gr82}.
In this paper a stability analysis is presented which shows that this element is stable
on a wide range of triangular grids.\footnote{\rblf{\textit{Added in 2020\/}: The
stability result  established here has since been generalised  by Boffi et al. 
 ({\it Journal of Scientific Computing, 52:383--400}, 2012)
 to cover Hood--Taylor triangular  and tetrahedral meshes  that are augmented by 
 adding piecewise constant functions  to the pressure space of continuous piecewise polynomials 
 of degree $k$ ($k\geq1$ in 2D and $k\geq 2$ in 3D) under the restriction that 
 every element has at least one vertex in the interior of the domain.}}
Details of the formulation of the continuous and discrete equations for Stokes flow
and the implications of the analysis of Stokes flow for Navier--Stokes flow will not be
given here. Only sufficient detail to introduce the notation is presented; for further
details see Girault \& Raviart~\cite{GR86}.

The Stokes equations, in weak form, may be written
\begin{align}
\text{find } (u,q)\in (H^1_0(\Omega))^2\times \Lto(\Omega) &\text{ such that}\nonumber\\[3pt]
\nu (\nabla u, \nabla v)_{\Omega}  - (\div v, q)_{\Omega} &= (f,v)_{\Omega} 
\quad \forall\, v \in (H^1_0(\Omega))^2, \label{eq:01} \\
(\div u, p)_{\Omega} &= 0 \quad \forall\, p \in  \Lto(\Omega), \label{eq:02}
\end{align}
where
\begin{enumerate}[(i)]
\item $H^1_0(\Omega) = \{ f \mid f\in H^1(\Omega), f=0 \hbox{ on } \partial\Omega\}$, 
\item $\Lto(\Omega) =  \{ f \mid f\in L^2(\Omega), (f, 1)_{\Omega} =0 \}$,
\item $(f, g)_{\Omega} = \int_\Omega f g \,d\Omega$.
\end{enumerate}
The discrete analogue of (\ref{eq:01}) and (\ref{eq:02}) in the finite element subspaces
$V^h\subset (H^1_0(\Omega))^2$ and $P^h\subset  \Lto(\Omega)$ is given by
\begin{align}
\text{find } (u^h,q)^h\in V^h \times P^h \text{ such that}& \nonumber\\[3pt]
\nu (\nabla u^h, \nabla v^h)_{\Omega}  - (\div v^h, q^h)_{\Omega} &= (f,v^h)_{\Omega} 
\quad \forall\, v^h \in V^h, \label{eq:03}\\
(\hbox{div\,} u^h, p^h)_{\Omega} &= 0 \quad \forall\, p^h \in P^h.  \label{eq:04}
\end{align}
The essential result for stability and convergence is the discrete `inf--sup' condition
\begin{equation}\label{eq:05}
\inf_{p\in P^h} \sup_{v\in V^h}
\left\{ \frac{(\div v, p)_{\Omega}}{| v |_{1,\Omega}\,  \| p \|_{L^2(\Omega)} } \right\}
\geq \beta >0
\end{equation}
with $\beta$ independent of $h$.
In general, such a condition can only be satisfied if all triangular elements
satisfy a regularity condition of the form\footnote{\rblf{\textit{Added in 2020\/}: This condition may be
relaxed. Certain approximation methods  (including Taylor--Hood) are known to be inf--sup stable
on highly stretched grids.} }
\begin{equation}\label{eq:06}
h_{\triangle} \leq \sigma \rho_{\triangle}
\end{equation}
where $\sigma>1$ and the parameters $h_{\triangle}$ and $ \rho_{\triangle}$
are respectively the diameter of element $\triangle$ and the diameter of the largest
circle inside $\triangle$.

For the Taylor--Hood element on a triangulation $T^h$ of $\Omega$, $P^h$ and $V^h$
are given by
\begin{align}
P^h &=  \{ p\in \Lto(\Omega) \cap C(\overline{\Omega}) \mid p\in P_1(\triangle)\ 
\forall\, \triangle \in T^h\} \label{eq:07} \\
V^h &=  \{ v \in (H^1_0(\Omega))^2  \mid v_i\in P_2(\triangle) \text{ for } i=1,2 
\text{ and } \forall\, \triangle \in T^h\}, \label{eq:08}
\end{align}
where $P_k(\triangle)$ is the set of all polynomials of degree less than or
equal to $k$ in  $\triangle$. The fact that the  Taylor--Hood element satisfies the
`inf--sup' condition \eqref{eq:05} was originally shown by Bercovier \& Pironneau~\cite{BP79},
although the results can now be proved quite easily using the patch ideas of
Boland \& Nicolaides~\cite{BN83} or equivalently Stenberg~\cite{St84}.

For the element discussed here, the space $V^h$ is the same and given by \eqref{eq:08} but
$P^h$ is defined by
\begin{equation}\label{eq:09}
P^h =  \{  p \mid  p=p_0 +p_1,\  p_0\in \Lto(\Omega),\
p_0\in P_0(\triangle),\
p_1\in \Lto(\Omega) \cap C(\overline{\Omega}),\
p_1\in P_1(\triangle), \forall\, \triangle \in T^h\}.
\end{equation}

\section{Stability of patches of elements}\label{sec:2}

Proving the discrete `inf--sup'  condition became relatively easy only after the ideas
of locally stable patches of elements were developed. The continuous linear plus
constant pressure element does not fit neatly into the local analysis of Stenberg~\cite{St84}
nor Boland \& Nicolaides~\cite{BN83} because on any patch there are always two distinct
ways of producing a constant function in the pressure space (i.e., constant at the
vertex nodes and zero at the centroids or zero at the vertices and constant at the
centroids). The analysis presented here is in the spirit of the
approach by Stenberg~\cite{St84}.\footnote{\rblf{\textit{Added in 2020\/}: In retrospect, a simpler and 
more elegant way of establishing stability would be  to employ  the construction used in the analysis
 of the Taylor--Hood element in Girault \& Raviart~\cite[pp.176--180]{GR86} together
 with the overlapping patch  framework  developed  by
Stenberg in his  follow-up paper ({\it Mathematics of Computation, 54:495--508}, 1990).}}

The notation $\esM$ is used for the class of patches of elements topologically equivalent to the patch of elements $M$.
By a patch of elements it is understood that it is a union of elements, each of which has at least one side in
common with another element of the patch. For the precise definition, see Stenberg~\cite{St84}.
By way of illustration  we note that the
patches~1.2, 1.3, 1.4 in figure~\ref{fig:01} are topologically equivalent but they are not equivalent to the patch~1.1.
Indeed all patches of 3 elements are topologically equivalent to either 1.1 or 1.2. The only constraint on $\esM$ is
that all elements must satisfy the regularity constraint \eqref{eq:06} for some value of $\sigma$.

\begin{figure}[htb]
\begin{center}
\includegraphics[width=.6\textwidth]{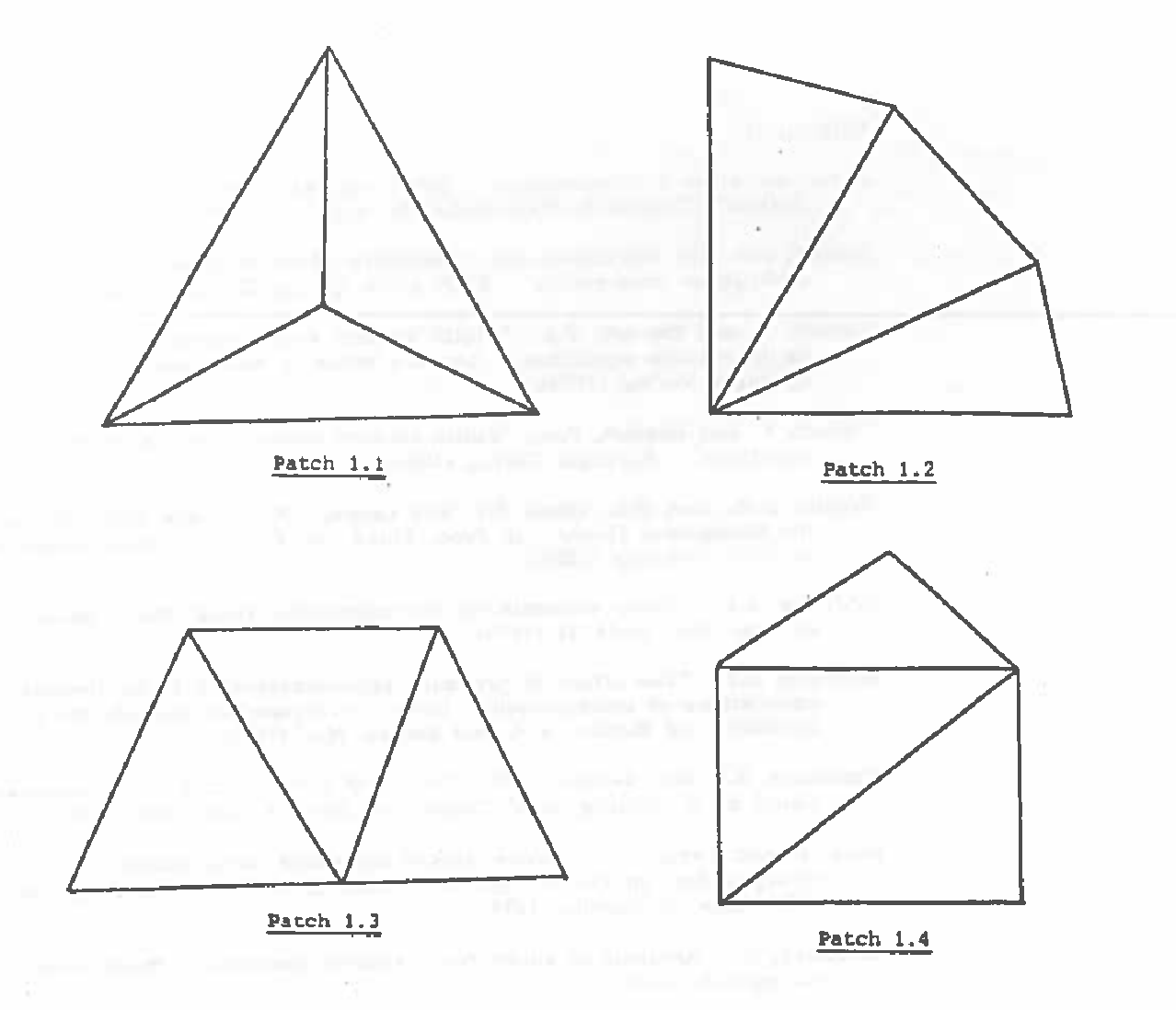}
\end{center}
\caption{Three-element patches.}\label{fig:01}
\end{figure}

To prove patch stability for a class of patches $\esM$, we first consider a typical patch $M \in \esM$. For this patch we define
\begin{alignat}2
&\text{(i)}   &\VhM &= \{v \in (H^1_0(M))^2 \mid v_i \in P_2(\triangle) \text{ for }
 i = 1,2 \text{ and } \forall\, \triangle \in M \},\label{eq:10}\\
&\text{(ii) } &\PhM &= \{p \mid p=p_0+p_1,\ p_0 \in P_0(\triangle)\ \forall\, \triangle \in M,\
p_1 \in C(\overline{M}),\ p_1 \in P_1(\triangle)\ \forall\,  \triangle \in M \}. \label{eq:11}
\end{alignat}
The first step in proving patch stability is to find a subspace $\RhM$ of $\PhM$ that satisfies the condition
\begin{equation}\label{eq:12}
\inf_{p\in \PhM} \sup_{v\in \VhM}
\left\{ \frac{(\div v, p)_{M}}{|v|_{1,M}\, \|p\|_{L^2(M)} } \right\} \geq \beta_M >0.
\end{equation}
The condition \eqref{eq:12} is equivalent to the condition
\begin{equation}\label{eq:13}
p \in \RhM,\quad (\div v, p)_M = 0\quad \forall\, v \in \VhM \quad\Longrightarrow\quad p=0.
\end{equation}
This condition can be established by looking at the null space of the matrix $B$ defined by
\begin{equation}\label{eq:14}
v \in \VhM,\quad p \in \PhM, \quad (\div v, p)_M = \underline{v}^T B \underline{p},
\end{equation}
where $\underline{v}$ is the vector of nodal coordinates determining the finite element function $v \in \VhM$ and 
$\underline{p}$ is the vector of nodal coordinates determining $p \in \PhM$. 
We will then show that the constraints to produce $\RhM$ from
$\PhM$ annihilate this null space. Such a process will establish \eqref{eq:12} for the particular patch $M$ and with the
particular definition of $\RhM$.

Both of the three-element patches 1.1 and 1.2 are important in subsequent sections of this paper and we will 
define an $\RhM$ and show that the constraints annihilate the null space in both cases.

\subsection{The patch 1.1}\label{sec:2.1}

A typical patch of this type, type~1, is illustrated\footnote{\rblf{\textit{Added in 2020\/}: 
The hand-drawn figures are reproduced here exactly as in the original report.}} in figure~\ref{fig:02}.
For this patch we define
\begin{equation}\label{eq:15}
\RhM = \{p \mid p \in \PhM,\ p_0 \in \Lto(M),\ p_1 \in \Lto(M) \}.
\end{equation}
\begin{figure}[htb]
\begin{center}
\includegraphics[width=.7\textwidth]{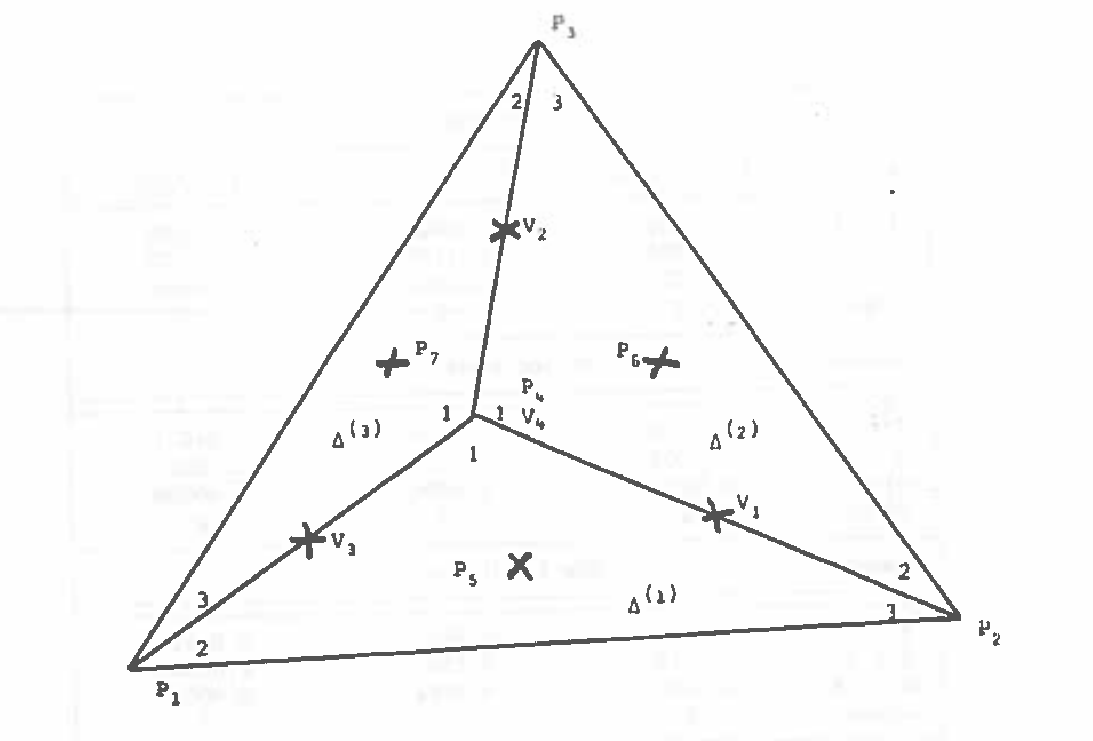}
\end{center}
\caption{A patch of type 1.}\label{fig:02}
\end{figure}
The patch has three elements $(\tri{i})^3_{i=1}$ and we have
\begin{align*} 
\underline{v}^T & = \{(V_1)_1, (V_1)_2, (V_2)_1, (V_2)_2, (V_3)_1, (V_3)_2, (V_4)_1, (V_4)_2, \},\\
\underline{p}^T & = \{P_1, P_2, \dots, P_7\},
\end{align*}
with $B$ an $8\times7$ matrix. We denote by $|\tri{i}|$ and $|M|$ the area of $\tri{i}$ and $M$
respectively and by the two-dimensional vector $b^{(i)}_j$ in $\tri{i}$ the usual values
\begin{equation*} 
(b^{(i)}_1)_1 = y^{(i)}_2 - y^{(i)}_3,\quad (b^{(i)}_1)_2 = x^{(i)}_3 - x^{(i)}_2, \text{ etc.}
\end{equation*}
with the local nodes of each element illustrated in figure~\ref{fig:02}. The matrix $B$ is given by
\begin{equation}\label{eq:16}
\frac16 
\begin{bmatrix}
-\bb12      & \bb21+\bb11 & -\bb23     & \bb22+\bb13 & -4\bb12 & -4\bb23 & 0      \\[4pt]
-\bb33      & -\bb22      & \bb31+\bb21& \bb32+\bb23 & 0       & -4\bb22 & -4\bb33\\[4pt]
\bb11+\bb31 & -\bb13      & -\bb32     & \bb12+\bb33 & -4\bb13 & 0       & -4\bb32\\[4pt]
0           & 0           & 0          & 0           & -\bb11  & -\bb21  & -\bb31
\end{bmatrix}
\end{equation}
and has null space spanned by
\begin{equation}\label{eq:17}
\begin{split}
S^T_1 &= (1,1,1,1,0,0,0),\\
S^T_2 &= (0,0,0,0,1,1,1).
\end{split}
\end{equation}
We denote by $S$ the $7\times2$ matrix, the first column of which is $S_1$ and the second is $S_2$. The constraints 
that give $\RhM$ from $\PhM$ are
\begin{enumerate}[(i)]
\item $\int_M p_0\,d\Omega = 0$,
\item $\int_M p_1\,d\Omega = 0$,
\end{enumerate}
which constrain the vector $\underline{p}$ by
\begin{enumerate}[(i)]
\item $P_5|\tri1| + P_6|\tri2| + P_7|\tri3| = 0$,
\item $P_1(|\tri1|+ |\tri3|)+ P_2(|\tri2|+|\tri1|) + P_3(|\tri3|+|\tri2|) + P_4|M| = 0$,
\end{enumerate}
which we write as
\begin{equation}\label{eq:18}
H \underline{p} = \underline{0},
\end{equation}
where $H$ is a $2\times7$ matrix. The only vector in the null space of $B$ that satisfies both constraints is the 
zero vector if the matrix
\begin{equation}\label{eq:19}
C_M= H S
\end{equation}
is of full column rank. Here $C_M$ is given by
\begin{equation}\label{eq:20}
\begin{bmatrix}
 0   & |M| \\
3|M| & 0
\end{bmatrix}.
\end{equation}
Clearly \eqref{eq:20} is nonsingular, therefore for any patch $M$ of type~1 with $\RhM$ given by \eqref{eq:15}, 
the inequality \eqref{eq:12} is satisfied provided $\RhM$ is nonempty. In section~\ref{sec:3} we will construct 
functions belonging to $\RhM$.

\subsection{The patch 1.2}\label{sec:2.2}

A typical patch of this type, type 2, is illustrated in figure~\ref{fig:03} with $\underline{v}$ a four-dimensional 
vector and $\underline{p}$ eight dimensional.
\begin{figure}[htb]
\begin{center}
\includegraphics[width=.5\textwidth]{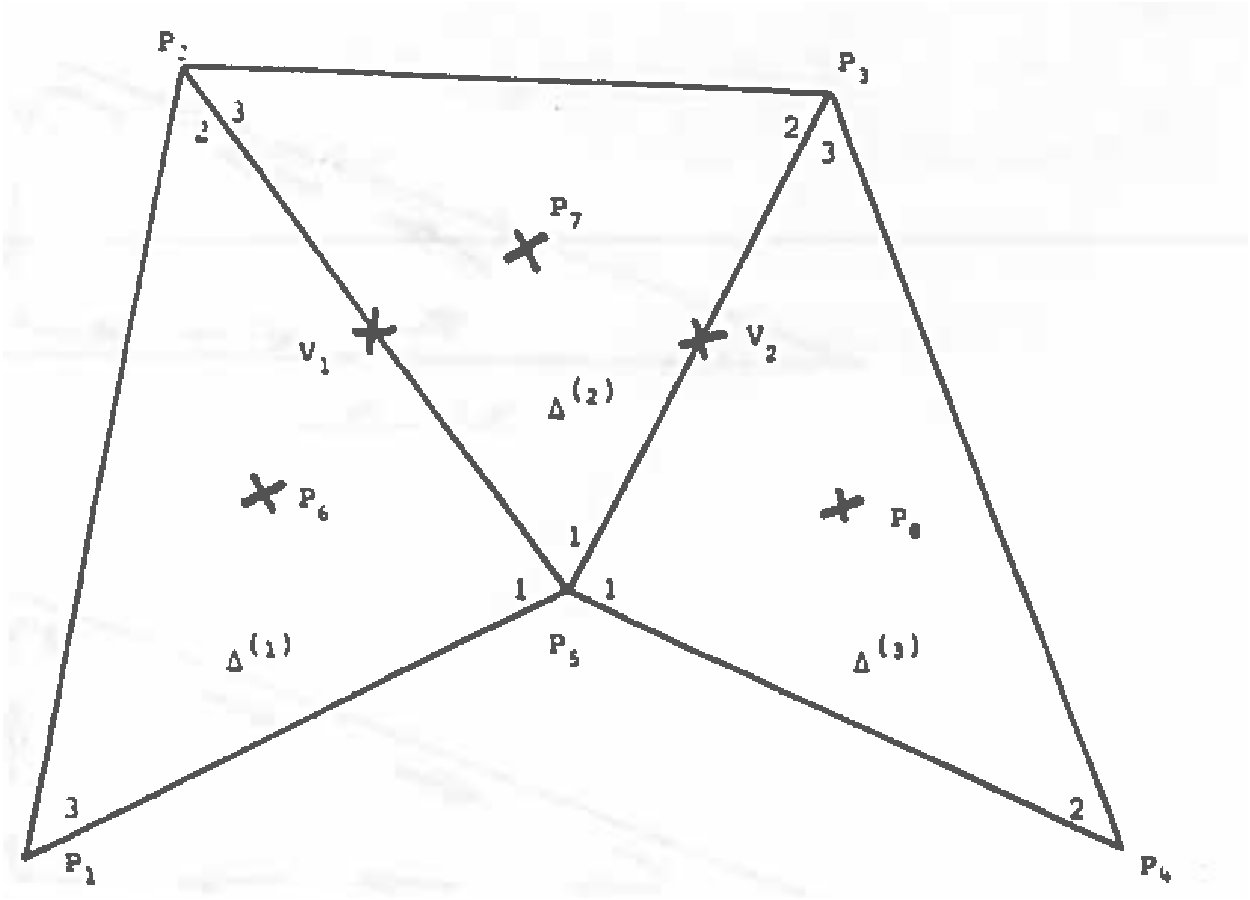}
\end{center}
\caption{A patch of type 2.}\label{fig:03}
\end{figure}
The matrix $B$ is the $4\times8$ matrix
\begin{equation}\label{eq:21}
\frac16
\begin{bmatrix}
-\bb13 & \bb21+\bb11 & -\bb22      & 0      & \bb23+\bb12 & -4\bb13 & -4\bb22 & 0      \\[4pt]
0      & -\bb23      & \bb31+\bb21 & -\bb32 & \bb33+\bb22 & 0       & -4\bb23 & -4\bb32
\end{bmatrix}
\end{equation}
which has a null space spanned by
\begin{equation}\label{eq:22}
\begin{split}
S^T_1 &=(1,1,1,1,1,0,0,0),\\
S^T_2 &=(0,0,0,0,0,1,1,1),\\
S^T_3 &=(4,0,0,0,0,-1,0,0),\\
S^T_4 &=(0,0,0,4,0,0,0,-1).
\end{split}
\end{equation}
Here we note that the matrix $S$ is an $8\times4$ matrix, the $i$th column of which is $S_i$. It is not possible
for the two constraints to produce $\RhM$ given by \eqref{eq:15} to give a matrix $C_M$ of full column rank
(because with these constraints $C_M$ is a $2\times4$ matrix). Thus we need to apply further constraints
on $\PhM$ to annihilate the null space. We define $\RhM$ by
\begin{equation}\label{eq:23}
\begin{split}
\RhM = \{p \mid p \in \PhM,\ p_0 \in \Lto(M),\ p_1 \in \Lto(M), \\ \qquad\qquad\qquad  p \in \Lto(\tri{i})\;
\forall\, \tri{i} \text{ with two sides on the boundary of }& M \}.
\end{split}
\end{equation}
Thus the constraints are
\begin{enumerate}[(i)]
\item $\int_M p_0 \, d\Omega = 0 \Longrightarrow P_6|\tri1| + P_7|\tri2| + P_8|\tri3| = 0$,
\item $\int_M p_1 \, d\Omega = 0 \Longrightarrow P_1|\tri1| + P_2(|\tri1|+|\tri2|) + P_3(|\tri2|+|\tri3|) 
         + P_4|\tri3| + P_5|M| = 0$,
\item $\int_{\tri1} p \, d\Omega = 0 \Longrightarrow P_1 + P_2 + P_5+ 3P_6 = 0$,
\item $\int_{\tri3} p \, d\Omega = 0 \Longrightarrow P_3 + P_4 + P_5 + 3P_8 = 0$,
\end{enumerate}
giving us the constraint equation of the form \eqref{eq:18}. Thus, here,
\begin{equation}\label{eq:24}
C_M = H S =
\begin{bmatrix}
0    & |M| & -|\tri1| & -|\tri3|\\[4pt]
3|M| & 0   & 4|\tri1| & 4|\tri3|\\[4pt]
3    & 3   & 1        &0        \\[4pt]
3    & 3   & 0        & 1
\end{bmatrix}.
\end{equation}
Clearly, \eqref{eq:24} is nonsingular, therefore for any patch $M$ of the type~2 with $\RhM$ given by \eqref{eq:23}, 
the inequality \eqref{eq:12} is satisfied provided $\RhM$ is nonempty. In section~\ref{sec:4} 
we will construct functions belonging to $\RhM$ for this patch.

\subsection{Stability over classes of topologically equivalent patches}\label{sec:2.3}

In this section we shall use the important result in Stenberg \cite{St84}, who observed that, 
given a class $\esM$ of topologically equivalent patches which satisfy \eqref{eq:12} for every $M \in \esM$, then 
the value of
\begin{equation}\label{eq:25}
\beta_M = \inf_{p\in\RhM} \sup_{v\in\VhM}
\left\{ \frac{(\div v, p)_M}{|v|_{1,M}\, \|p\|_{L^2(M)}} \right\}
\end{equation}
is independent of a transformation of the form
\begin{equation}\label{eq:26}
\tilde{\underline{x}} = A(\underline{x} - \underline{\alpha}),
\end{equation}
where $A$ and $\alpha$ are fixed. This, together with the regularity constraint \eqref{eq:06} allows us to represent the
whole range of values of $\beta_M$ for $M\in\esM$ as a  function over a compact set $R$, every point of which
represents a patch (or indeed many patches) belonging to $\esM$. Thus there exists $\beta>0$ such that
\begin{equation}\label{eq:27}
\beta = \min_{R} (\beta_M) = \min_{M\in\esM} (\beta_M)
\end{equation}
for which
\begin{equation}\label{eq:28}
\inf_{p\in\RhM} \sup_{v\in\VhM}
\left\{ \frac{(\div v, p)_M}{|v|_{1,M}\, \|p\|_{L^2(M)}} \right\} \geq \beta > 0
\end{equation}
for every $M\in\esM$ with $\beta$ independent of $M$ and $h$ and depending only on the topology of $\esM$ 
and $\sigma$. Thus here, for a given $\sigma$, we have two classes $\beps^\sigma_1$ and $\beps^\sigma_2$ 
of elements topologically equivalent to the patches of type~1, and~2 respectively, and the inequality \eqref{eq:28} 
holds for these two classes with
different values of $\beta$. We will actually use an equivalent form of \eqref{eq:28}, namely for each $p\in\RhM$, 
there exists $v_M\in\VhM$
such that
\begin{equation} \label{eq:29}
\begin{split}
(\div v_M, p)_M &\geq \beta \|p\|_{L^2(M)}^2,\\[-2pt]
|v_M|_{1,M} &\leq C \|p\|_{L^2(M)},
\end{split}
\end{equation}
or, for $\gamma > 0$,
\begin{equation}  \label{eq:30}
\begin{split}
(\div (\gamma v_M), p)_M &\geq \beta\gamma \|p\|_{L^2(M)}^2,\\[-2pt]
|(\gamma v_M)|_{1,M} &\leq C\gamma \|p\|_{L^2(M)},
\end{split}
\end{equation}
with $\beta$ and $C$ depending only on the class of patches and on $\sigma$.

\section{Global stability of grids made up of patches of the type 1}\label{sec:3}

In order to use the results of the previous section we construct an operator $\PiM p$ by
\begin{align} 
\PiM p & =  \frac{1}{|M|}  \int_M p  \,d\Omega = k_0 + k_1,\label{eq:31}\\
& \text{where}\quad k_0 =  \frac{1}{|M|} \int_M {p_0} \,d\Omega  ,\quad k_1 
=   \frac{1}{|M|} \int_M {p_1} \,d\Omega. \label{eq:32}
\end{align}
Thus, since
\begin{equation*}
\int_M (p_0 - k_0) \, d\Omega = 0, \quad \int_M (p_1 - k_1)\,d\Omega = 0
\end{equation*}
then using the argument from Stenberg \cite{St84} we can establish the following theorem and corollary.

\begin{thm}\label{thm:3.1}
For every $p\in\PhM$,
\begin{enumerate}[\upshape(A)]
  \item $p - \PiM p \in \RhM$, $\RhM$ defined by \eqref{eq:15},
  \item $(\div v, \PiM p) = 0\quad \forall\, v \in \VhM$.
\end{enumerate}
\end{thm}

\begin{cor}[Corollary to Theorem~\ref{thm:3.1}]
For each $p\in \PhM$, there exists  $v_M \in \VhM$ such that
\begin{equation*}
\begin{split}
\big(\div v_M, p\big)_M &\geq \beta_1 \|p - \PiM p\|_{L^2(M)}^2,\\
|v_M|_{1,M} &\leq C_1 \|p - \PiM p\|_{L^2(M)},
\end{split}
\end{equation*}
where $\beta_1$ and $C_1$ are independent of $p$ and $v_M$ but are dependent on $\sigma$.
\end{cor}

The proof of global stability on a grid made up of patches of the type~1 now follows   that given by 
Stenberg \cite{St84} or Boland \& Nicolaides \cite{BN83}. Moreover, we have the optimal rates of convergence 
to the solution $(u,q)$ of the Stokes equation\footnote{\rblf{\textit{Added in 2020\/}: Assuming additional
smoothness  (that is, $H^3$ regularity) of the target solution.}}, namely
\begin{equation}\label{eq:33}
|u-u^h|_{1,\Omega} + \|q-q^h\|_{L^2(\Omega)} \leq C h^2 (|u|_{3,\Omega} + \|q\|_{2,\Omega})
\end{equation}
where $(u^h,q^h)$ is the numerical solution in $V^h\times P^h$.

\section{Further results on the patch of the type 2}\label{sec:4}

Before we establish global stability and optimal rates of convergence, there are further results required for a
patch of this type. Let $M$ be a patch of the type~2, we construct an operator $\Pi^{(2)}_M$ on the set $\PhM$ by
\begin{equation}\label{eq:34}
\Pi^{(2)}_M p  = 
\begin{cases}
k^{(2)} \enskip \text{in } \tri2 ,\\
      4k^{(2)} \big(1-3L^{(i)}\big) -3k^{(i)}\big(1-4L^{(i)}\big) \enskip \text{in } \tri{1}  \text{ and }  \tri{3} ,
\end{cases}
\end{equation}
where
\begin{enumerate}[(i)]
\item $\tri{i}$ is defined in figure~\ref{fig:03},
\item $L^{(i)}$ is the areal coordinate in $\tri{i}$, $i = 1$ and $3$, equal to 1 at
the intersection of the two sides on the boundary $\partial M$ of $M$,\footnote{\rblf{\textit{Added in 2020\/}: 
Thus with the numbering shown in figure~3, we have $L^{(i)}=L_3$ in ${\tri{1}}$ and  $L^{(i)}=L_2$ in ${\tri{3}}$.}}
\item \hfill \makebox[0pt][r]{%
\begin{minipage}[b]{\textwidth}
\begin{equation}\label{eq:35}
k^{(i)} = \frac{1}{|\tri{i}|} \int_{\tri{i}} \! {p\,d\Omega}\quad   \text{for } i=1,2,3.
\end{equation}
\end{minipage}}
\end{enumerate}

\begin{thm}\label{thm:4.1}
For every $p \in \PhM$ then
\begin{enumerate}[\upshape(A)]
\item \hfill \makebox[0pt][r]{%
\begin{minipage}[b]{\textwidth}
\begin{equation}\label{eq:36}
p - \Pi^{(2)}_M \mu \in \RhM,\quad \RhM \textup{ defined by \eqref{eq:23}},
\end{equation}
\end{minipage}}
\item \hfill \makebox[0pt][r]{%
\begin{minipage}[b]{\textwidth}
\begin{equation}\label{eq:37}
\big(\div v, \Pi^{(2)}_M p\big)_M = 0\quad \forall\, v \in \VhM.
\end{equation}
\end{minipage}} 
\end{enumerate}
\end{thm}

\begin{proof}
\noindent (A)\enskip 
Clearly $p = \Pi^{(2)}_M p \in \PhM$, thus we have to show that $p-\Pi^{(2)}_M p$ satisfies the
constraints from $\PhM$ to $\RhM$. We write
\begin{align}
\left[\Pi^{(2)}_M p\right]_0  &=
\begin{cases}
 (1-\alpha)k^{(2)} \enskip \text{in } \tri2 \\
      (1-\alpha)k^{(2)} + 3\big(k^{(2)}-k^{(i)}\big) \enskip \text{in } \tri{i},\quad i = 1 \text{ and } 3,
\end{cases} \label{eq:38}
\\
\left[\Pi^{(2)}_M p\right]_1 & =
\begin{cases}
 \alpha k^{(2)} \enskip \text{in } \tri2 \\
       \alpha k^{(2)} - 12\big(k^{(2)}-k^{(i)}\big)L^{(i)} \enskip \text{in } \tri{i},\quad i = 1 \text{ and } 3,
\end{cases}\label{eq:39}
\end{align}
with
\begin{equation}\label{eq:40}
\alpha = \frac{\int_M p_1\,d\Omega + 4 (k^{(2)}-k^{(1)})|\tri1| + 4 (k^{(2)}-k^{(3)})|\tri3|}
{k^{(2)}|M|}
\end{equation}
if $k^{(2)}\neq 0$. (If $k^{(2)}=0$ then the definition of $\big[\Pi^{(2)}_M p\big]_i$ for $i = 0,1$ is independent of $\alpha$.)

We now observe that
\begin{alignat*}2
&\int_M p_0 - \left[\Pi^{(2)}_M p\right]_0 \,d\Omega = 0,&\quad & \int_M p_1 - \left[\Pi^{(2)}_M p\right]_1 \,d\Omega = 0, \\
&\int_{\tri1} p - \Pi^{(2)}_M p \,d\Omega = 0,&\quad & \int_{\tri3} p - \Pi^{(2)}_M p \,d\Omega = 0.
\end{alignat*}
Thus we have established part (A).

\medskip
\noindent (B)\enskip
Taking the left-hand side of \eqref{eq:37},
\begin{equation}\label{eq:41}
\begin{split}
\big(\div v, \Pi^{(2)}_M p\big)_M & = \big(\div v, k^{(2)}\big)_M + 3\big(k^{(2)}-k^{(1)}\big)\big(\div v,(1-4L^{(1)})\big)_{\tri1} \\
  & \quad + 3\big(k^{(2)}-k^{(3)}\big) \big(\div v,(1-4L^{(3)})\big)_{\tri3}.
\end{split}
\end{equation}
The first term in \eqref{eq:41} is zero because $v\in\VhM$. If $L^{(1)}$ is $L^{(1)}_3$ in $\tri1$ 
then $v=V_1(4L^{(1)}_2 L^{(1)}_1)$, thus we see immediately\footnote{\rblf{\textit{Added in 2020\/}: This can be verified
by direction computation;  $\int_{\triangle} L_i (1-4L_j) \, d {\triangle} =0$ whenever $i\neq j$.}} that the second
term in \eqref{eq:41} is zero and so therefore is the third. 
\end{proof}

\begin{cor}[Corollary to Theorem \ref{thm:4.1}]
For each $p\in\PhM$ there exists  $v_M \in \VhM$ such that
\begin{align}\label{eq:42}
(\div v_M, p)_M & \geq \beta_2 \|p-\mu_M\|_{L^2(\tri2)}^2 \, ,\\
\label{eq:43}
|v_M|_{1,M} & \leq C_2 \|p-\mu_M\|_{L^2(\tri2)}\, ,
\end{align}
with
\begin{equation}\label{eq:44}
\mu_M = \frac{1}{|\tri2|} \int_{\tri2} p\,d\Omega,
\end{equation}
where $\beta_2$ and $C_2$ depend only on the regularity constant $\sigma$.
\end{cor}

\begin{proof}
By Theorem~\ref{thm:4.1} and  inequalities \eqref{eq:30}, for each $p\in \PhM$ there exists  $\tilde{v}_M$
such that, for $\gamma>0$,
\begin{gather*}
(\div (\gamma\tilde{v}_M), p)_M = \big(\div(\gamma\tilde{v}_M), p-\Pi^{(2)}_M p\big) 
\geq \beta_2 \gamma \,  \big\|p-\Pi^{(2)}_M p\big\|_{L^2(M)}^2,  \\
|\gamma\tilde{v}_M|_{1,M} \leq c_2\gamma \big\|p-\Pi^{(2)}_M p\big\|_{L^2(M)}.
\end{gather*}
The result follows by choosing
\begin{equation}\label{eq:45}
\gamma = \big\|p-\Pi^{(2)}_M p\big\|_{L^2(\tri2)} \bigg/ \big\|p-\Pi^{(2)}_M p\big\|_{L^2(M)},
\end{equation}
noting that $\sigma<\gamma<1$, and choosing
$v_M = \gamma \tilde{v}_M$.
\end{proof}

\section{Global stability of grids made up of patches of the type 2}\label{sec:5}

We assume that $\Omega$ is a polygonal region which has been triangulated into $N$ triangular
elements $(E^{(i)})$. We further assume that each element $E^{(i)}$ sits inside an extended patch $M_i$
of the type~2 as element $\tri2$. Firstly we note that these patches overlap. Secondly we note that this 
assumption does exclude some grids of triangles, namely those which either
\begin{enumerate}[(a)]
\item contain triangular elements with two sides on the boundary, or
\item contain a patch of the form 1.1 with that patch having a side on the boundary.
\end{enumerate}
We can see this by looking at figure~\ref{fig:04}, elements $E^{(2)}$ to $E^{(12)}$ all
sit as element $\tri2$ of a patch of the type~2 but elements $E^{(1)}$ and $E^{(13)}$ do not.

\begin{figure}[htb]
\begin{center}
\includegraphics[width=.4\textwidth]{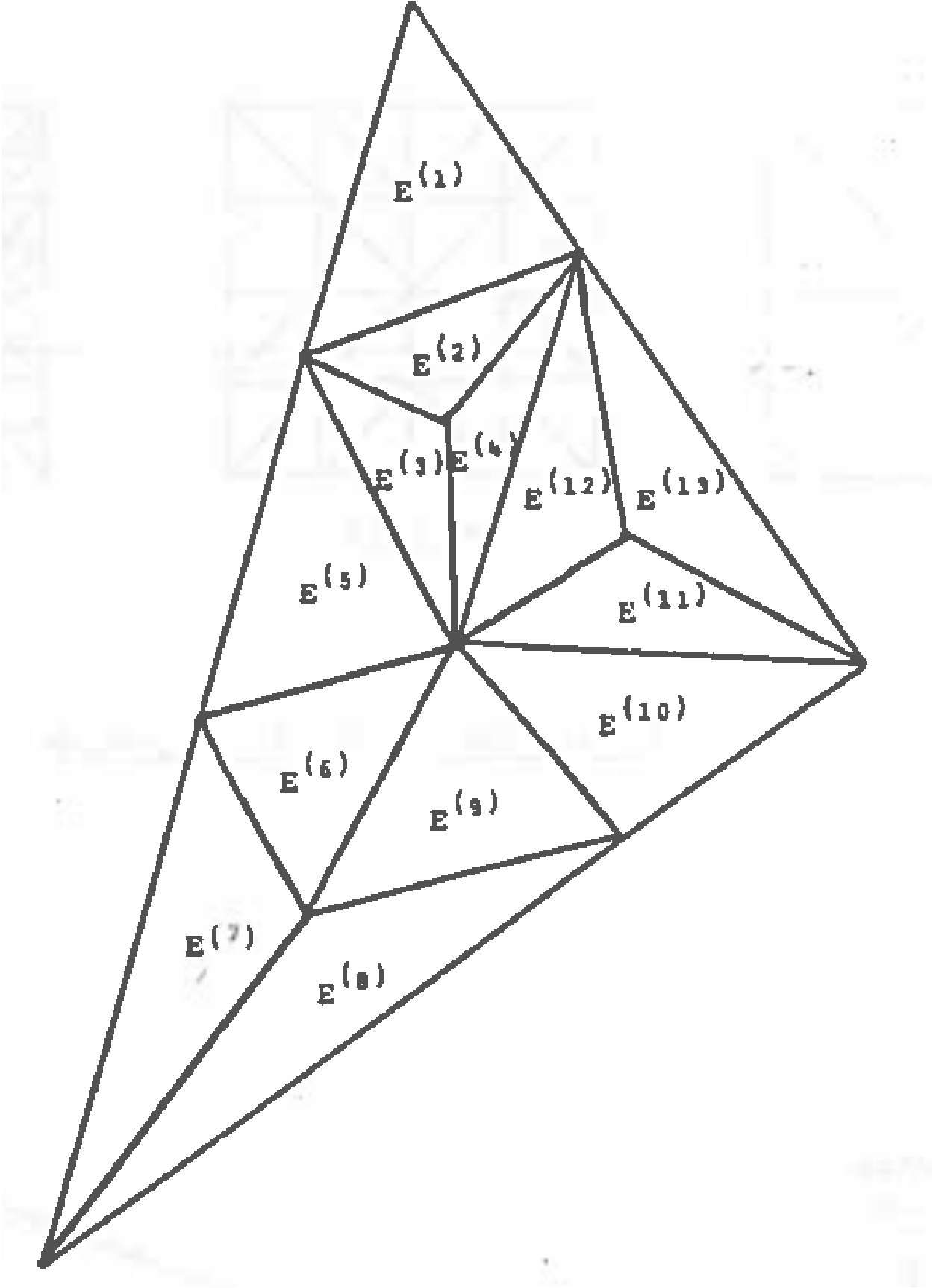}
\end{center}
\caption{Sample triangulation illustrating the restrictions on the grid.}\label{fig:04}
\end{figure}

We consider $\mu\in\Lto(\Omega)$ such that $\mu$ is constant in each $E^{(i)}$, Girault \& 
Raviart \cite{GR79} show that for each $\mu$ there exists  $ \tilde{v} \in V^h$ such that
\begin{equation}\label{eq:46}
\begin{split}
(\div\tilde{v}, \mu)_\Omega & = \|\mu\|_{L^2(\Omega)}^2,\\
|\tilde{v}|_{1,\Omega}    & \leq \tilde{c} \|\mu\|_{L^2(\Omega)}. 
\end{split}
\end{equation}
Let $p\in P^h$, we define $\mu\in\Lto(\Omega)$ such that $\mu$ in the element $E^{(i)}$
is the constant value
\begin{equation}\label{eq:47}
\mu = \frac{1}{|E^{(i)}|} \, \int_{E^{(i)}} p\,d\Omega \quad \text{in } E^{(i)} \text{ for } i=1,\ldots,N;
\end{equation}
thus by the corollary to Theorem~\ref{thm:4.1}, for each $p\in P^h$ there exists  $v_{M_i} \in V^h_{M_i}$
such that
\begin{align}\label{eq:48}
(\div v_{M_i}, p)_{M_i} & \geq \beta_2 \|p-\mu\|_{L^2(E^{(i)})}^2,\\
|v_{M_i}|_{1,M_i}    & \leq C_2 \|p-\mu\|_{L^2(E^{(i)})}.
\label{eq:49}
\end{align}
We note that $\beta_2$ and $C_2$ do not depend on the particular patch $M$ and depend only
on the regularity constant $\sigma$. Thus, for each $p\in P^h$ there exists $\tilde{v}\in V^h$ such that
\begin{align}\label{eq:50}
(\div\tilde{v}, p)_\Omega & \geq \tilde{\beta} \|p-\mu\|_{L^2(\Omega)}^2,\\
|\tilde{v}|_{1,\Omega}    & \leq \tilde{c} \|p-\mu\|_{L^2(\Omega)},
\end{align}
where
\begin{alignat}{2}
\nonumber
\text{(a)}&\enskip & \tilde{v} &= \sum^N_{i=1} \tilde{v}_i,
\\
\label{eq:52}
\text{(b)}&\enskip & \tilde{v}_i& = 
\begin{cases}
v_{M_i} & \text{in } M_i,\\
0 &\text{elsewhere in } \Omega,
\end{cases}
\\
\nonumber
\text{(c)}&\enskip & \tilde{C} &= C_2,\quad \tilde{\beta} = \beta_2.
\end{alignat}

Following Stenberg \cite{St84} the inequalities \eqref{eq:46} and \eqref{eq:50} establish that
for each $p\in P^h$ then
\begin{equation}\label{eq:53}
v = \tilde{v} + \left(\frac{2\tilde{\beta}}{1+  \tilde{C} ^2} \right)\tilde{v}
\end{equation}
satisfies
\begin{align}\label{eq:54}
(\div v, p)_{\Omega} & \geq \left(\frac{\tilde{\beta}}{1+  \tilde{C}^2}\right) \|p\|_{L^2(\Omega)}^2,\\
|v|_{1,\Omega}    & \leq \left(\tilde{C} + \frac{2\tilde{\beta} \tilde{C} }{1+  \tilde{C} ^2}\right) \|p\|_{L^2(\Omega)}.
\label{eq:55}
\end{align}
Thus we have global stability and optimal rates of convergence on all grids that satisfy the regularity constraint
\eqref{eq:06} and restrictions on the triangles mentioned at the beginning of this section.

It is only the former of these restrictions, namely that we must triangulate `into the corners' that is an essential
restriction. By including patches of the form 1.1 in the above argument then the second restriction can be removed but 
now $\tilde{\beta} = \max(\beta_1, \beta_2)$ and $\tilde{C} = \min(C_1, C_2)$. Thus we have established optimal 
convergence rates on all grids of triangles provided the grid has been triangulated into the corners.
Further discussion of this topic when the grid has not been triangulated into the corners is given in the appendix.

\section{Numerical examples}\label{sec:6}

In this section we shall consider two numerical examples. The first is a simple test problem to demonstrate
that optimal convergence rates are achieved and the second is an example where the Taylor--Hood element gives
poor results but the linear plus constant pressure element, which we shall call the 
LC element\footnote{\rblf{\textit{Added in 2020\/}: The mixed approximation method is referred to 
$\bP_2$--$\bP_{-1*}$  in the book by  Elman et al. [{\it Finite Elements and Fast Iterative Solvers},
Oxford University Press, 2014),
and as the  {\it enhanced Hood--Taylor} scheme in the book by Boffi et al. [{\it Mixed Finite 
Element Methods and Applications}, Springer, 2013].}}, gives relatively 
good results.

\subsection{Testing rates of convergence}\label{sec:6.1}

The first test problem is one proposed by Griffiths \& Mitchell~\cite{GM79}. It is an enclosed flow problem
(namely a Stokes flow) in the unit square with solution
\begin{equation}\label{eq:56}
\begin{split}
v_x & = -20xy^3,\\
v_y & = 5y^4 - 5x^4,\\
p   & = -60x^2y + 20y^3 +5.
\end{split}
\end{equation}
Typical grids for this test problem are illustrated in figure~\ref{fig:05} and the solutions (i.e.\ norms
\begin{figure}[htb]
\begin{center}
\includegraphics[width=.6\textwidth]{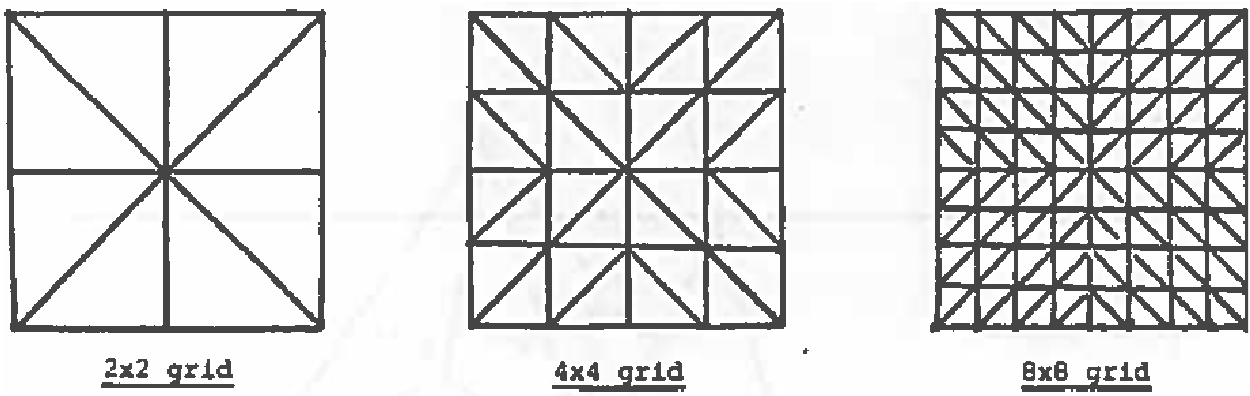}
\end{center}
\caption{Typical grids for the Griffiths problem.}\label{fig:05}
\end{figure}
of the errors) are presented in table~\ref{tab:01} with the results for the LC element compared with the Taylor--Hood 
element and the Raviart bubble element; see Girault \& Raviart \cite{GR86}. It can be seen that the error for the LC
and Taylor--Hood elements are comparable and both are much smaller than the Raviart bubble element for this problem.
Moreover, for all three elements, the optimal convergence rates are observed.

\begin{table}[htb]
\caption{Computed errors for the Griffiths test problem}\label{tab:01}
\begin{center}
\begin{tabular}{l ccc}
\hline
Grid    & $\|p\|_{L^2(\Omega)/R}$ & $\| \vec{v}\|_{H^1(\Omega)}$ & $\| \vec{v} \|_{L^2(\Omega)}$ \\[5pt]
\hline
& \multicolumn{3}{c}{Taylor--Hood element} \\\hline
$4\times4$   & 0.4283 & 0.4802 & 0.01669\\
$8\times8$   & 0.0975 & 0.1189 & 0.00239\\
$16\times16$ & 0.0233 & 0.0296 & 0.00029\\
Order & ${\sim}2$ & ${\sim}2$ & ${\sim}3$\\\hline
& \multicolumn{3}{c}{LC element} \\\hline
$4\times4$   & 0.4878 & 0.4865 & 0.01637\\
$8\times8$   & 0.1009 & 0.1190 & 0.00237\\
$16\times16$ & 0.0233 & 0.0296 & 0.00029\\
Order & ${\sim}2$ & ${\sim}2$ & ${\sim}3$\\\hline
& \multicolumn{3}{c}{Raviart bubble element} \\\hline
$4\times4$   & 1.5469 & 0.6834 & 0.02416\\
$8\times8$   & 0.4314 & 0.1797 & 0.00342\\
$16\times16$ & 0.1142 & 0.0462 & 0.00043\\
Order & ${\sim}2$ & ${\sim}2$ & ${\sim}3$\\
\hline
\end{tabular}
\end{center}
\end{table}

\subsection{Illustrating difficulties with the Taylor--Hood element}\label{sec:6.2}

The solution of a non-Newtonian fluid in the volume of revolution of the region illustrated in
figure~\ref{fig:06}, with the boundary conditions given in the figure, can be reduced to solving the
following set of equations
\begin{subequations}\label{eq:57}
\begin{align}
\label{eq:57a} 
\frac{1}{\Rey}\left( \frac{\pr^2 V_r}{\pr r^2} + \frac{\pr^2 V_r}{\pr z^2} + \frac{1}{r}\frac{\pr V_r}{\pr r} 
- \frac{V_r}{r^2}\right) - \frac{\pr p}{\pr r}& 
= - \Nn \frac{ V^2_\theta}{r} + \Ne\bigg[\left(\frac{\pr V_\theta}{\pr r} - \frac{V_\theta}{r}\right)^2 +
 \left(\frac{\pr V_\theta}{\pr z}\right)^2\bigg],  \\
\label{eq:57b} 
\frac{1}{\Rey}\left( \frac{\pr^2 V_\theta}{\pr r^2} + \frac{\pr^2 V_\theta}{\pr z^2} 
+ \frac{1}{r}\frac{\pr V_\theta}{\pr r} - \frac{V_\theta}{r^2}\right)& =0,\\
\label{eq:57c} 
\frac{1}{\Rey}\left( \frac{\pr^2 V_z}{\pr r^2} + \frac{\pr^2 V_z}{\pr z^2} + \frac{1}{r}\frac{\pr V_z}{\pr r}\right) 
- \frac{\pr p}{\pr z} & =0,\\
\label{eq:57d}
\frac{\pr V_r}{\pr r} + \frac{V_r}{r} + \frac{\pr V_z}{\pr z} & = 0,
\end{align}
\end{subequations}
after a number of assumptions have been made; further details of which are given by Tidd~\cite{Ti87}. For a 
Newtonian fluid the parameter $\Ne=0$ and for a non-Newtonian fluid this parameter gives a measure of 
the non-Newtonian effects. We note that
equation~\eqref{eq:57b} is independent of $V_r$, $V_z$, $p$, $\Ne$ and decouples from the other three equations 
whereas equations~\eqref{eq:57a}, \eqref{eq:57c},~\eqref{eq:57d} represent a Stokes flow problem in $(r,z)$ 
coordinates. This can be solved for the two cases $(\Nn=1, \Ne=0)$ and $(\Nn=0, \Ne=1)$ and the particular solution 
required can be obtained by selecting the required ratio of these two intermediate solutions.

\begin{figure}[htb]
\begin{center}
\includegraphics[width=.6\textwidth]{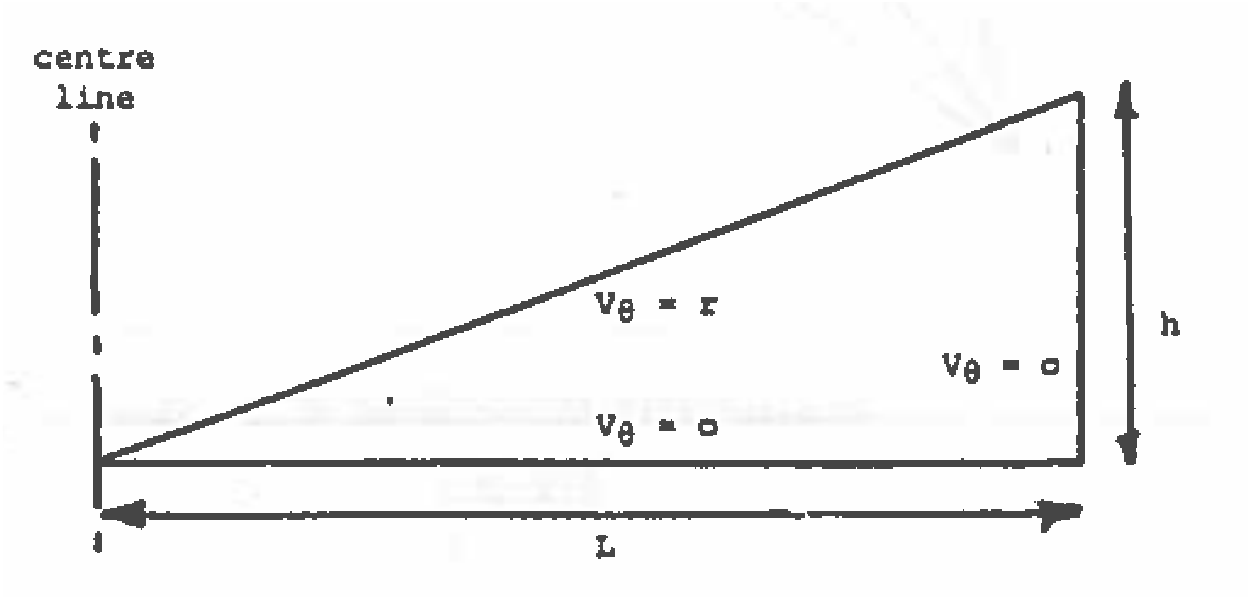}
\end{center}
\caption{The second test problem $(L=1.0, h=0.1)$.}\label{fig:06}
\end{figure}

\begin{figure}[!htb]
\begin{center}
\includegraphics[width=.3\textwidth]{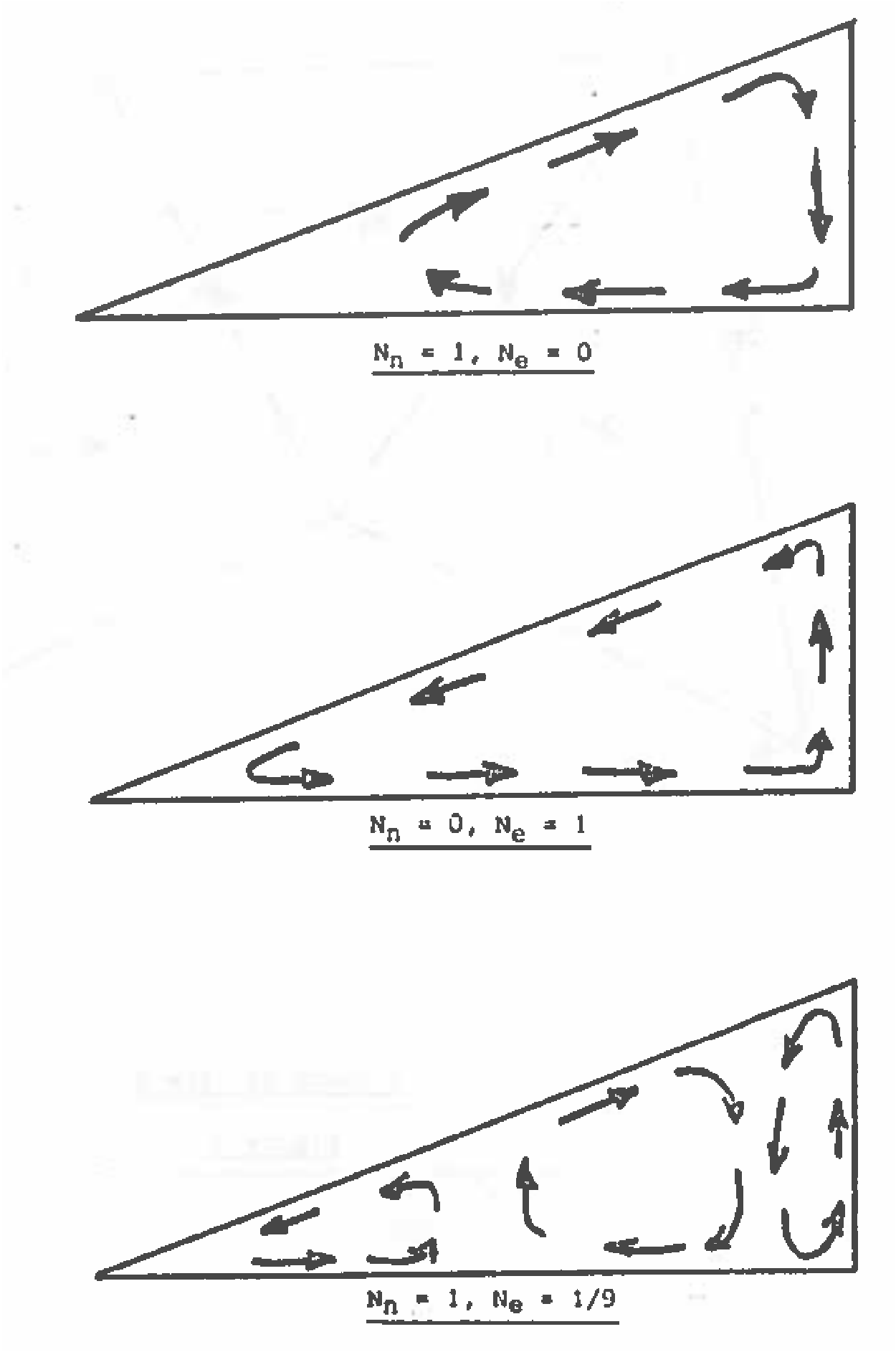}
\end{center}
\caption{Illustration of  the secondary recirculations for the cases of interest $(\Rey=10)$.}\label{fig:07}
\end{figure}

The main flow in this problem is a  swirling ($V_\theta$) flow with the $V_r$ and $V_z$ representing
secondary flows. An indication of the secondary flows for the three cases
\begin{enumerate}[(a)]
\item $(\Nn=1, \Ne=0)$,
\item $(\Nn=0, \Ne=1)$,
\item $(\Nn=1, \Ne=1/9)$, 
\end{enumerate}
is given in figure~\ref{fig:07}. The three secondary recirculations of (c) have been observed by 
Hoppmann \& Baronet \cite{HB65} and
it is in attempting to model these three recirculations that we find that the Taylor--Hood element gives a
very poor solution even on a highly refined grid.

\begin{figure}[!htb]
\begin{center}
\includegraphics[width=.3\textwidth]{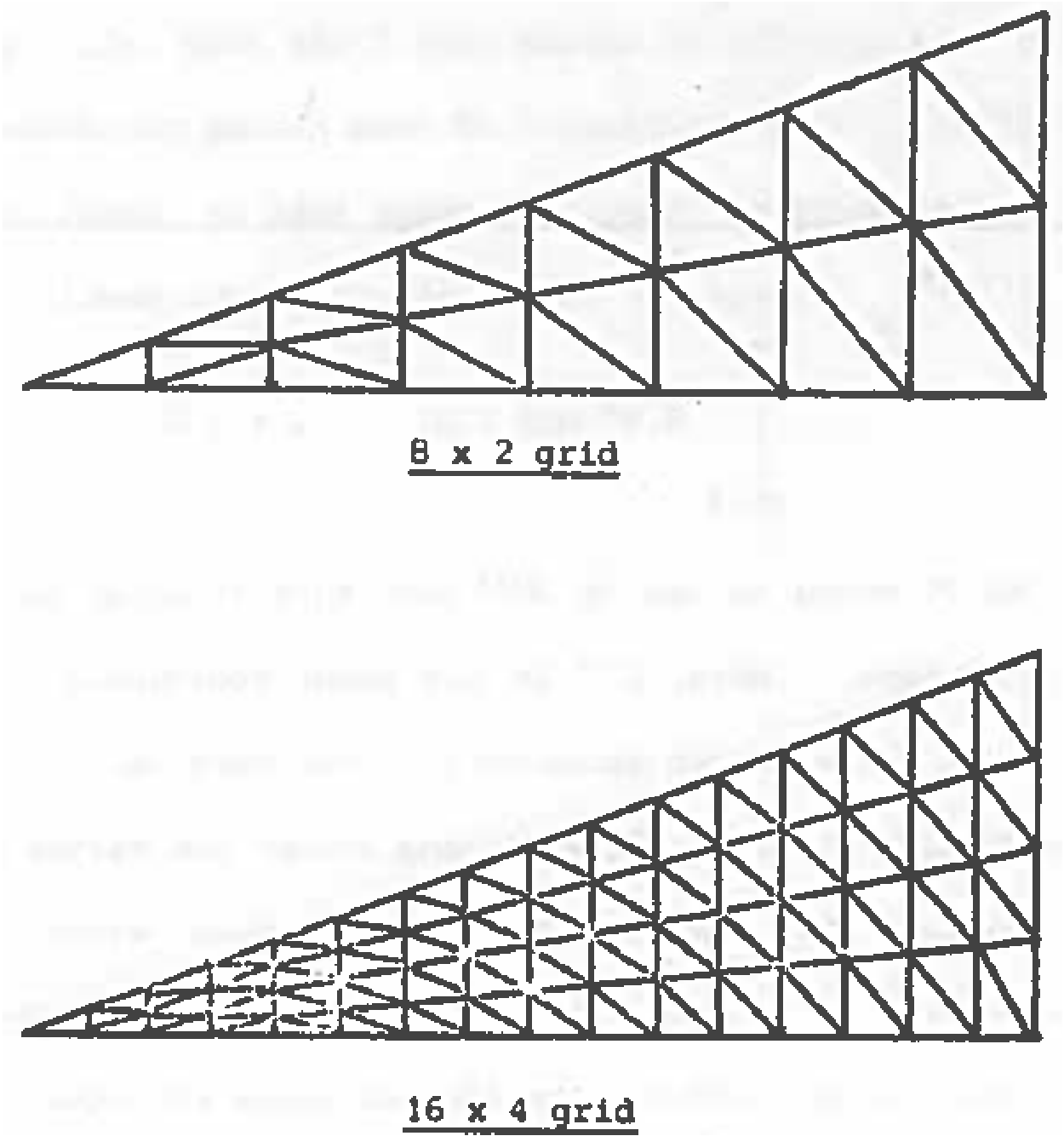}
\end{center}
\caption{Typical grids for the second test problem.}\label{fig:08}
\end{figure}

Typical grids used are illustrated in figure~\ref{fig:08}. The $V_\theta$ problem, namely equation~\eqref{eq:57b} was 
solved on each of the grids. On a given grid the relevant $V_\theta$ numerical solution was used on the 
right-hand side of equation~\eqref{eq:57a}.

We find that the Taylor--Hood and LC elements gave essentially the same results for the case $(\Nn=1, \Ne=0)$ but very
different answers for the case $(\Nn=0, \Ne=1)$ with the LC solutions giving far more consistency from one grid to the next.
Moreover, for the case $(\Nn=1, \Ne=1/9)$, where we are expecting to observe three recirculations, the Taylor--Hood
element gives only one complete recirculation with velocities in almost random directions over almost half the region of the
problem, namely $0\leq r \leq 0.5$, even on a highly refined $64\times16$ grid. However, the LC element resolves
all three recirculations on a $16\times4$ grid and gives very good consistency between the $32\times8$ and $64\times16$ grids.

We note that the above problem does not fall into the analysis in section~\ref{sec:5}. Not only is the problem in $(r,z)$
coordinates but also we have not triangulated into all the corners. By modifying the grid so that we do triangulate into
all the corners the solutions obtained do not change in any significant way.

In order to obtain a solution with the LC element when we have not triangulated into the corners it is necessary either to
fix a pressure in each element with two sides on the boundary or to make the centroid pressure in this element equal to
the centroid pressure in the element with a shared side. In fact these two strategies only affect the pressure solution
in the element with two sides on the boundary but the analysis of the latter approach is more simple, and is discussed in
the appendix.

\section{Conclusion}\label{sec:7}
We have established stability and convergence for the LC element on a wide range of grids. The proof had an
interesting feature, namely enclosing the patches under consideration in larger and overlapping extended patches.
This idea may prove useful in establishing convergence for other elements for which (div-)stability is
difficult to obtain.\footnote{\rblf{\textit{Added in 2020\/}: The only technical requirement is that
each element in the subdivision belongs at most to a finite number $N$ of macro-element patches with $N$ 
independent of $h$, see Remark 8.5.6 in  Boffi et al. [{\it Mixed Finite Element Methods and Applications}, 
Springer, 2013].}}

When using the element for an enclosed flow problem it is necessary to fix two pressure values, one must be a 
centroid pressure and the other a vertex pressure.

Finally, the LC element clearly has some interesting features that make it worthy of consideration. In particular the fact
that it is locally incompressible has been demonstrated to be useful in practice.

\subsection*{Acknowledgements}

We would like to acknowledge the support of the SERC who provided one of us (RWT) with a grant to collaborate with
Professor Nicolaides at Pittsburgh, and to Professor Nicolaides for his help in getting the work in this
report concluded. We would also like to acknowledge the contribution of David Tidd, an SERC research student,
who calculated the numerical results for the second example.



\appendix
\setcounter{equation}{0}
\renewcommand{\theequation}{A\arabic{equation}}

\section{Appendix. Restoring stability when the grid is not triangulated into the corners}\label{sec:a}

In this appendix we assume that $\Omega$ has been split up into $N$ elements $(E^{(i)})^N_{i=1}$
with the first $n$ of them having two sides on the boundary. A simple numerical experiment shows that we cannot 
use the element LC for $(E^{(i)})^N_{i=1}$ because the continuity equation namely
\begin{equation}\label{eq:a01}
\int_{E^{(i)}} \div v^h_M \, d\Omega = 0, \quad 1\leq i\leq n,
\end{equation}
with $M$ equal to one in $E^{(i)}$ and with $M$ equal to $L^{(i)}$ in $E^{(i)}$ are linearly dependent.
(Here, $L^{(i)}$ is the areal coordinate in $E^{(i)}$ equal to~$1$ at the vertex between two sides on the boundary
$\pr\Omega$ of $\Omega$). We can overcome this difficulty by arbitrarily assigning either the vertex pressure or the
centroid pressure (to zero or any other value) without affecting the velocity solution. It is interesting to note
that this strategy does not destroy the reason for introducing the element since the equations~\eqref{eq:a01} with $M$
equal to one or $L^{(i)}$ are merely multiples of each other and keeping either of them in the system of equations
ensures continuity over the element $E^{(i)}$. If we make the particular choice of removing the centroid pressure 
(i.e.~setting it to zero), then this is equivalent to using a Taylor--Hood element for $E^{(i)}$, $1\leq i\leq n$.
 It is surprising that we have not been able to obtain a satisfactory stability result for this strategy.

We could set this centroid pressure to any other value and it only affects the resulting numerical solution by changing the
pressure approximation in that element. The particular strategy that we analyse below is choosing the centroid pressure
in $E^{(i)}$to be equal to the centroid pressure in $E^{(n+i)}$, where $E^{(n+i)}$ is the element that
has a side in common with $E^{(i)}$. This is equivalent to using the pressure space
\begin{equation}\label{eq:a02}
\begin{split}
P^h = \big\{p \mid 
      p &=p_0+p_1;\quad
      p_0 \in P_0(E^{(i)} \cup E^{(n+i)})\quad \forall\, i=1,\ldots,n, \\
      p_0 &\in P_0(E^{(i)})\quad \forall\, i =  n+1,\ldots, N;
      p_0 \in \Lto(\Omega); \\
      p_1 &\in P_1(E^{(i)})\quad \forall\, i = 1, \ldots, N;\
      p_i \in \Lto(\Omega) \cap C(M) \big\}.
\end{split}
\end{equation}
Let $M$ be the three-element patch of type~3 illustrated in figure~\ref{fig:09} for which
\begin{equation}\label{eq:a03}
\begin{rcases}
\tri1 = E^{(j)}\\
\tri2 = E^{(n+j)}
\end{rcases}
\quad\text{for } 1 \leq j \leq n.
\end{equation}
\begin{figure}[htb]
\begin{center}
\includegraphics[width=.4\textwidth]{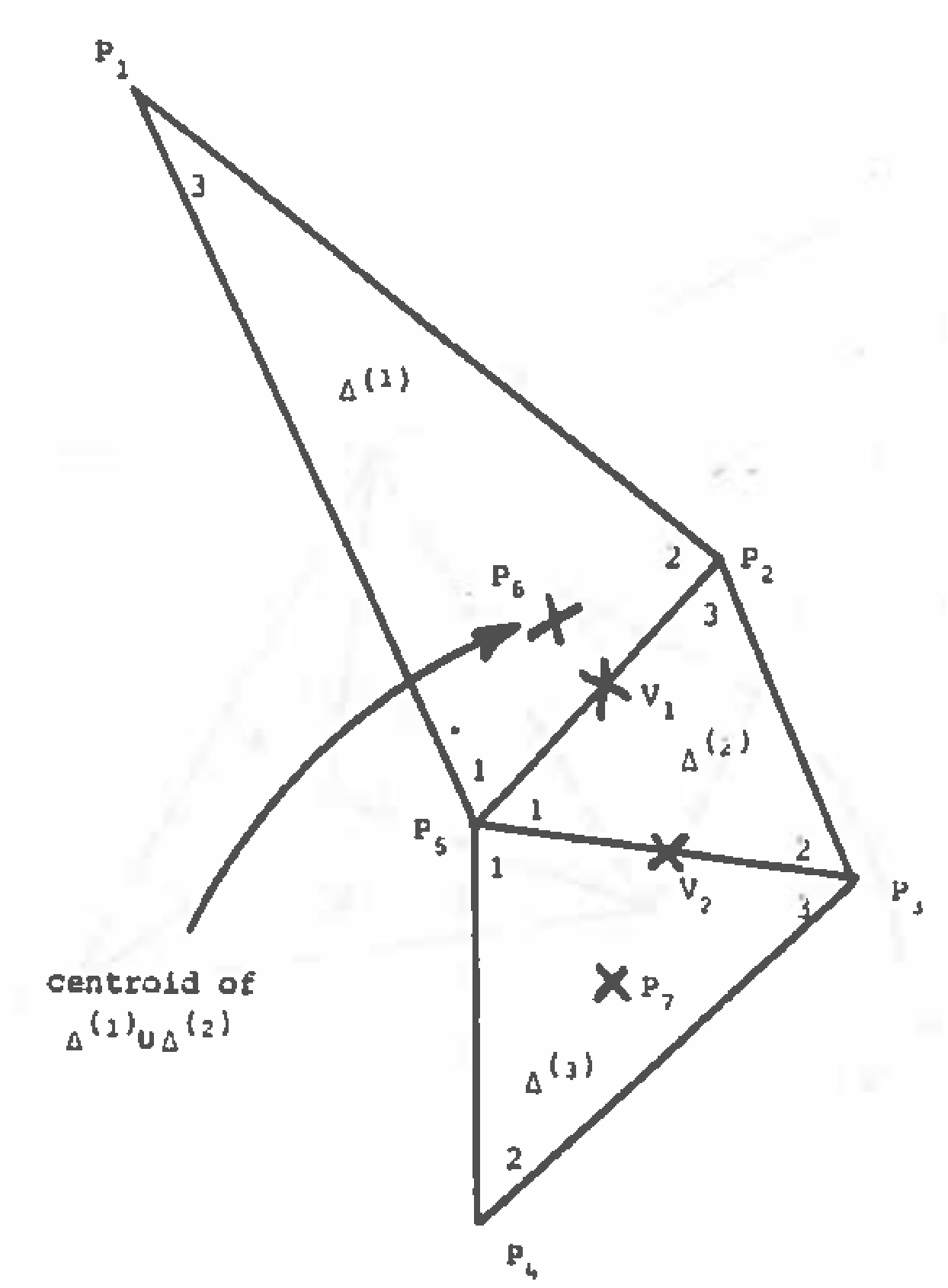}
\end{center}
\caption{A patch of type 3.}\label{fig:09}
\end{figure}

Using the notation of section~\ref{sec:2} then $\underline{v}$ is a four-dimensional vector and $\underline{p}$ a 
seven-dimensional vector and the space $\PhM$ for this patch is
\begin{equation}\label{eq:a04}
\begin{split}
\PhM = \big\{p =p_0+p_1 \mid 
       p_0 \in P^0(\tri1\cup\tri2),\ 
       p_0 \in P_0(\tri3),&  \\
       p_1 \in C(\overline{M}),\
       p_1 \in P_1(\tri{i}) \text{ for } i=1,2,3& \big\} .
\end{split}
\end{equation}
The matrix $B$ is the $4\times7$ matrix
\begin{equation}\label{eq:a05}
\begin{bmatrix}
-\bb13 & \bb21+\bb11 & -\bb22      & 0      & \bb23+\bb12 & 0       & 0      \\[4pt]
0      & -\bb23      & \bb31+\bb21 & -\bb32 & \bb33+\bb22 & -4\bb23 & -4\bb32
\end{bmatrix}
\end{equation}
which has null space spanned by
\begin{equation}\label{eq:a06}
\begin{split}
S^T_1 & = (1,1,1,1,1,0,0), \\
S^T_2 & = (0,0,0,0,0,1,1), \\
S^T_3 & = (0,0,0,4,0,0,-1),
\end{split}
\end{equation}
and we denote by $S$ the $7\times3$ matrix the $i$th column of which is $S_i$. The space $\RhM$ for this patch
is defined to be
\begin{equation}\label{eq:a07}
\RhM = \big\{p \mid 
       p = p_0+p_1 \in \PhM,\
       p_0 \in \Lto(M),\
       p_1 \in \Lto(M),\
       p \in \Lto(\tri3) \big\}.
\end{equation}
Thus the constraints from $\PhM$ to $\RhM$ are
\begin{enumerate}[(i)]
\item $\displaystyle \int_M p_0\, d\Omega = 0 \Longrightarrow P_6\big(|\tri1|+|\tri2|\big) + P_7|\tri3| = 0$,
\item $\displaystyle \int_M p_1\, d\Omega = 0 \Longrightarrow P_1|\tri1| + P_2\big(|\tri1|+|\tri2|\big)
+ P_3\big(|\tri2|+|\tri3|\big) + P_4|\tri3| +P_5|M| = 0$,
\item $\displaystyle \int_{\tri3} p\, d\Omega = 0 \Longrightarrow P_3+P_4+P_5+3P_7 = 0$,
\end{enumerate}
giving the constraint equation~\eqref{eq:23} and the matrix $C_M$ is given by
\begin{equation}\label{eq:a08}
C_M = H S = 
\begin{bmatrix}
0    & |M| & -|\tri3| \\[4pt]
3|M| & 0   & 4|\tri3| \\[4pt]
3    & 3   & 1
\end{bmatrix},
\end{equation}
which is clearly nonsingular.

We denote by $\beps^\sigma_3$ the class of patches of  type~3 illustrated in figure~\ref{fig:09}. Thus, for every
$M\in\beps^\sigma_3$ then for each $p\in\RhM$ (defined by \eqref{eq:a07}) there exists $v_M \in \VhM$
 (defined by \eqref{eq:10}) such that
\begin{equation}\label{eq:a09}
\begin{split}
(\div (\alpha v_M), p)_M & \geq \alpha \beta_3 \|p\|_{L^2(M)}^2, \\
|\alpha v_M|_{1,M} & \leq \alpha C_3 \|p\|_{L^2(M)},
\end{split}
\end{equation}
with $\beta_3$ and $C_3$ independent of the choice of $M\in\beps^\sigma_3$ but dependent on the 
regularity constant $\sigma$, where $\alpha$ is any positive constant.

Letting $M\in\beps^\sigma_3$, we construct an operator $\Pi^{(3)}_M$ on the set $\PhM$ (defined by \eqref{eq:a09}) by
\begin{equation}\label{eq:a10}
\Pi^{(3)}_M p =
\begin{cases}
k = k^{(1)}+k^{(2)} & \text{in } \tri1 \text{ and } \tri2,\\
4k\big(1-3L^{(3)}\big) - 3k^{(3)}\big(1-4L^{(3)}\big) & \text{in } \tri3,
\end{cases}
\end{equation}
where
\begin{enumerate}[(i)]
\item $\tri{i}$ are as shown in figure~\ref{fig:09},
\item $L^{(3)}$ is the areal coordinate in $\tri3$ that is equal to zero on the side in common with $\tri2$,
\item $k^{(i)}$ is defined as in \eqref{eq:35}.
\end{enumerate}

\begin{thm}\label{thm:a1}
For every $p\in\PhM$ then
\begin{enumerate}[\upshape(I)]
\item $p - \Pi^{(3)}_M p \in \RhM$,
\item $\big(\div v, \Pi^{(3)}_M p\big) = 0 \quad \forall\, v \in \VhM$
\end{enumerate}
(with $\PhM$, $\RhM$, $\VhM$ defined by \eqref{eq:a04}, \eqref{eq:a07}, \eqref{eq:10} respectively).
\end{thm}

\begin{proof}
To prove (I) we need to show two results. First, that
\begin{equation}\label{eq:a11}
\int_{\tri3} \big(p - \Pi^{(3)}_M p\big) \, d\Omega = 0.
\end{equation}
Second, that we can split $\big(p - \Pi^{(3)}_M p\big)$ into
\begin{equation}\label{eq:a12}
\big(p - \Pi^{(3)}_M p\big) = \left[p - \Pi^{(3)}_M p\right]_1 + \left[p - \Pi^{(3)}_M p\right]_0
\end{equation}
such that
\begin{enumerate}[(a)]
\item $\big[p - \Pi^{(3)}_M p\big]_0$ is constant in $\tri1\cup\tri2$, constant in $\tri3$ and satisfies
\begin{equation}\label{eq:a13}
\int_M \left[p - \Pi^{(3)}_M p\right]_0 \, d\Omega = 0,
\end{equation}
\item $\big[p - \Pi^{(3)}_M p\big]_1$ is linear in each $\tri{i}$, is continuous over $M$ and satisfies
\begin{equation}\label{eq:a14}
\int_M \left[p - \Pi^{(3)}_M p\right]_1 \, d\Omega = 0.
\end{equation}
\end{enumerate}
To do this we write
\begin{enumerate}[(i)]
\item $p = p_0 + p_1$, with $p_0$ and $p_1$ defined by \eqref{eq:a03},
\item 
$
\big[p - \Pi^{(3)}_M p\big]_1 = \begin{cases} p_1 - \alpha    &\quad \text{in } \tri1\cup\tri2\\
                               p_1 - \alpha - 12\big(k^{(3)}-k\big) L^{(3)} & \quad \text{in } \tri3,
\end{cases}$
\item $
\big[p - \Pi^{(3)}_M p\big]_0 =  \begin{cases}p_0 - k + \alpha    &\quad \text{in } \tri1\cup\tri2\\
                               p_0 - 4k + 3k^{(3)} + \alpha &\quad \text{in } \tri3,
\end{cases}$
\end{enumerate}
with $\alpha$ as yet undefined. 
Next, choosing $\alpha$ so that
\begin{equation}\label{eq:a15}
\alpha = \frac{k|M| + 3\big(k-k^{(3)}\big) |\tri3| -\int_M p_0\, d\Omega}{|M|},
\end{equation}
ensures that equation \eqref{eq:a13} is satisfied.
Moreover, since the definition \eqref{eq:a10} of $\Pi^{(3)}_M p$ ensures that
\begin{equation}\label{eq:a16}
\int_M \big(p - \Pi^{(3)}_M p\big) \, d\Omega = 0,
\end{equation}
then, for this value of $\alpha$, we see that equation \eqref{eq:a14} is also satisfied. 
Finally, the definition of $\Pi^{(3)}_M p$ ensures
that equation \eqref{eq:a11} holds. Thus we have established part (I) of the theorem. 
Part (II) of the theorem follows by 
the same argument as used in part (B) of theorem~\ref{thm:4.1}.
\end{proof}

\begin{cor}[Corollary to theorem \ref{thm:a1}]
For each $p\in\PhM$ there exists $v_M \in \VhM$ such that
\begin{align*}
(\div v_M, p)_M & \geq \beta_3 \|p-\mu_M\|_{L^2(\tri1\cup\tri2)}^2,\\
|v_M|_{1,M}      & \leq c_3 \|p-\mu_M\|_{L^2(\tri1\cup\tri2)},
\end{align*}
with
\begin{equation*}
\mu_M = \frac{\int_{\tri1\cup\tri2} p\, d\Omega}{|\tri1| + |\tri2|}  .
\end{equation*}
\end{cor}

\begin{proof}
The proof of this corollary is similar to the proof of the corollary to theorem~\ref{thm:4.1}.
\end{proof}

To establish global stability and optimal convergence rates we assume that $\Omega$ is a polygonal region 
which has been triangulated into $N$ triangular elements $\{E^{(i)}\}^N_{i=1}$. We assume that the first $n$
of these elements (with $n\ll N$) have two sides in common with the boundary $\pr\Omega$ of $\Omega$.
We denote by $E^{(n+1)}$ that element which has a side in common with $E^{(i)}$, $1\leq i \leq n$, and we
denote by $M_{n+1}$ the extended patch of the type~3 which contains $E^{(i)}$ as element $\tri1$ and
$E^{(n+i)}$ as element $\tri2$. We further assume that each element $E^{(i)}$, $2n+1 \leq i \leq N$, 
sits inside an extended patch $M_i$ of the type~2 as element $\tri2$. These assumptions exclude grids that
\begin{enumerate}[(a)]
\item contain two elements both of which have two sides on the boundary $\pr\Omega$ of $\Omega$ and which
 have a side in common with a single element of the grid,

\item contain patches of the type~1 with one side in common with boundary $\pr\Omega$ of $\Omega$.
\end{enumerate}

We define the pressure space $P^h$ by \eqref{eq:a02} and a function $\mu \in \Lto(\Omega)$ 
by
\begin{align*}
\mu  =  
\begin{cases}
\int_{E^{(i)}\cup E^{(n+i)}} p\, d\Omega \big/ {(|E^{(i)}|+|E^{(n+i)}|)} &\quad \text{in } E^{(i)} 
\cup E^{(n+i)} \quad \text{ for } i = 1,2,\dots,n \\[2pt]
    \int_{E^{(i)}} p\, d\Omega \big/ {|E^{(i)}|} &\quad \text{in } E^{(i)} \quad \text{ for } i = 2n+1,\dots, n.
\end{cases}
\end{align*}
The argument now follows that in section~\ref{sec:5} recognising that there are 
two types of patch involved and with
\begin{equation*}
\tilde{C} = \max (C_2, C_3), \quad \beta = \min (\beta_2, \beta_3).
\end{equation*}
Thus, we obtain global stability and optimal convergence rates 
(i.e.~inequality \eqref{eq:33}) on a wide range of grids 
even when the grid has not been triangulated into the corners.

To generalise the above argument to include all possible grids in this optimal convergence result we have to reconsider 
those grids that are excluded. The exclusion (a) above can be overcome by taking a three-element patch that looks 
like the patch 1.2 or 1.3 in which the term $p_0$ is constant throughout the patch. In practice, as one refines the 
grid to obtain more accurate results the necessity for having such elements as described in (a) is removed.

\begin{figure}[!htb]
\begin{center}
\includegraphics[width=.5\textwidth]{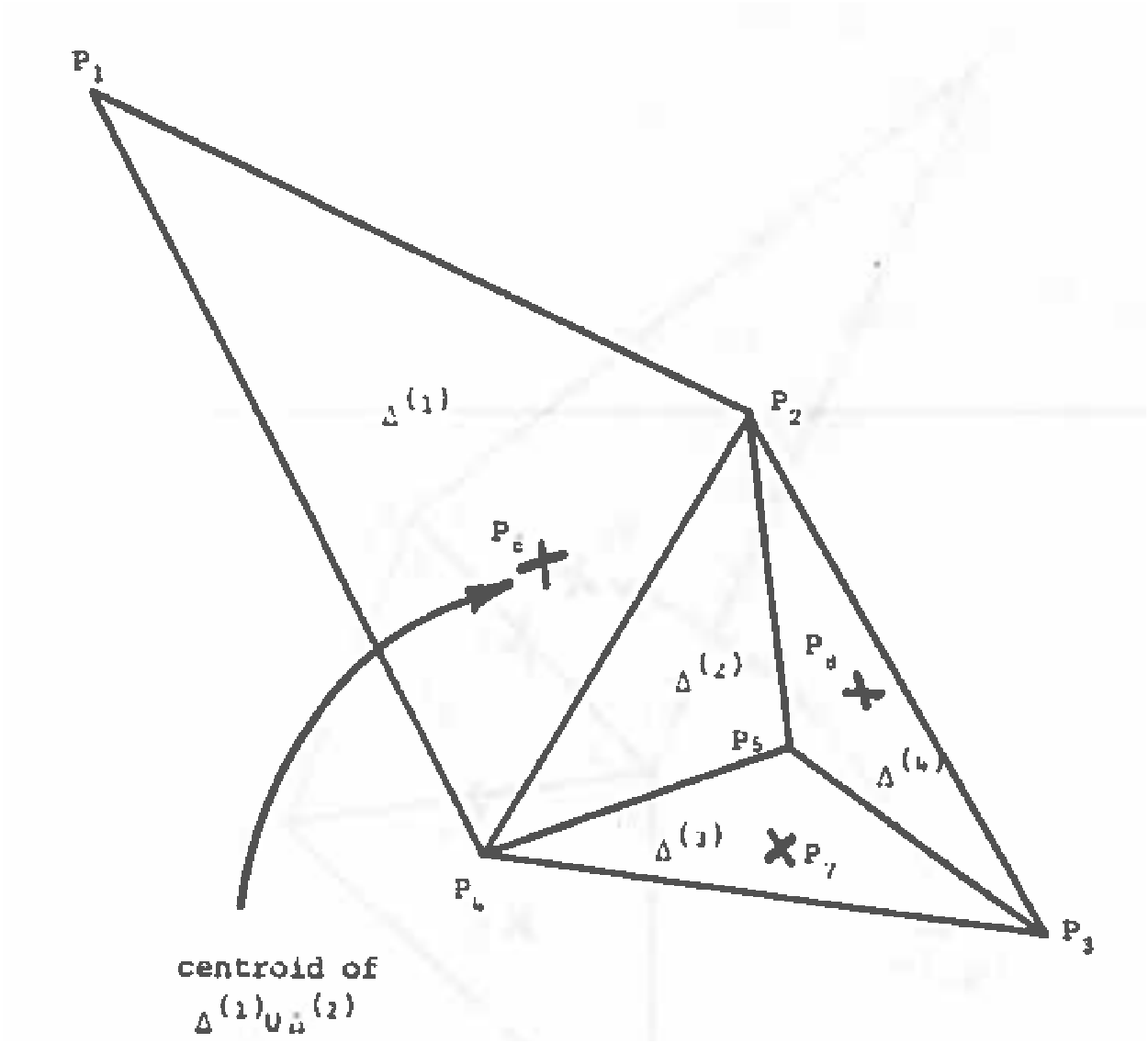}
\end{center}
\caption{A four-element patch.}\label{fig:10}
\end{figure}

The elements excluded by (b) can be included in the argument by using patches of the type~1 as mentioned in 
section~\ref{sec:5}. But to fully generalise the argument we need to also include the four-element patch illustrated
 in figure~\ref{fig:10} to cover the hypothetical possibility that for  one of the elements $\{E^{(i)}\}^{n}_{i=1}$, the
 element $E^{(n+i)}$ belongs to a patch of the type~1 which has oneside on the boundary.
  In this patch the term $p_0$ in the pressure space is constant throughout $\tri1\cup\tri2$ 
(or equivalently $E^{(i)}\cup E^{(n+i)}$).

\end{document}